\documentclass{amsart}

\usepackage{latexsym}
\usepackage{booktabs}

\usepackage{amssymb}
\usepackage{amsfonts}
\usepackage{amsmath}
\usepackage{amsthm}
\usepackage{rotating}
\usepackage{color}
\usepackage[pdftex,bookmarks,colorlinks]{hyperref} 
\hypersetup{colorlinks,%
                    citecolor=blue,%
                      filecolor=blue,%
                      linkcolor=red,%sub-orb 
                      urlcolor=blue}%
\usepackage[pdftex]{hyperref}

\newtheorem{theorem}{Theorem}[section]

\newtheorem{lemma}[theorem]{Lemma}
\newtheorem{proposition}[theorem]{Proposition}
 
\newtheorem{remark}[theorem]{Remark}
\newtheorem{example}[theorem]{Example}
\newtheorem{problem}[theorem]{Problem}
\newtheorem{construction}[theorem]{Construction}

\def\Ga{\Gamma}

\newcommand{\PGL}{{\mathrm{PGL}}}	

\newcommand{\Sym}{{\mathrm{Sym}}}
\newcommand{\Alt}{{\mathrm{Alt}}}
\newcommand{\Aut}{{\mathrm{Aut}}}

\newcommand{\calF}{{\mathcal{O}\mathcal{G}}}
\newcommand{\calAG}{\mathcal{AG}}

\newcommand{\Cdr}[1]{{\sf C}^{\rightarrow}_{#1}}
\def\cent#1#2{{\bf C}_{{#1}}({{#2}})}

\newcommand{\ov}{\overline}
\newcommand{\GD}{\mathcal{G}(\Delta)}

\title[Edge-transitive oriented graphs of valency four]
{Finite edge-transitive oriented graphs of valency four: a global approach}

\author[J. A. Al-bar, A. N. Al-kenani,
N. M. Muthana, C. E. Praeger, and Pablo Spiga]
{Jehan A. Al-bar, Ahmad N. Al-kenani,\\ 
Najat Mohammad Muthana,  
Cheryl E. Praeger,\\
and Pablo Spiga}
\address[First four authors]{King Abdulaziz University\\
Jeddah\\
Saudi Arabia} 

\address[Cheryl E. Praeger]
{Also affiliated with: Centre for the Mathematics of Symmetry and Computation\\
School  of Mathematics and Statistics M019\\
The University of Western Australia\\
35 Stirling Highway\\
Crawley, WA 6009\\
Australia}

\address[Pablo Spiga]
{Dipartimento di Matematica e Applicazioni, University of Milano-Bicocca, Via Cozzi 53, 20125 Milano, Italy}

\thanks{This project was funded by the Deanship of Scientific Research (DSR), 
King Abdulaziz University, Jeddah, under grant no. HiCi/H1433/363-1. The authors, 
therefore, acknowledge with thanks DSR technical and financial support.
}

\email[Jehan A. Al-bar]{jalbar@kau.edu.sa; jaal[underscore]bar@hotmail.com}
\email[Ahmad N. Al-kenani]{analkenani@kau.edu.sa; aalkenani10@hotmail.com}
\email[Najat M. Muthana]{nmuthana@kau.edu.sa; }
\email[Najat M. Muthana (second email)]{najat[underscore]muthana@hotmail.com}
\email[Cheryl E. Praeger]{cheryl.praeger@uwa.edu.au}
\email[Pablo Spiga]{pablo.spiga@unimib.it}


\begin{document}

\begin{abstract}
We develop a new framework for analysing finite connected, oriented graphs of valency $4$,
which admit a vertex-transitive and edge-transitive group of automorphisms
preserving the edge orientation. We identify a sub-family of `basic' graphs 
such that each graph of this type is a normal cover of at least one basic graph. 
The basic graphs either admit an edge-transitive group of automorphisms that is 
quasiprimitive or biquasiprimitive on vertices, or admit an
(oriented or unoriented) cycle as a normal quotient. We anticipate that each of these additional properties 
will facilitate effective further analysis, and we demonstrate that this is so for 
the quasiprimitive basic graphs. Here we obtain strong restirictions on the group
involved, and construct several infinite families of such graphs which, 
to our knowledge, are different from any recorded  in the literature so far. 
Several open problems are posed in the paper.
\end{abstract}

\maketitle

\section{Introduction}

We initiate a new approach to studying finite connected oriented graphs of valency four, 
which admit a vertex-transitive and edge-transitive group of automorphisms preserving edge 
orientations. We make a normal quotient reduction leading to what we call basic graphs which 
either admit quasiprimitive or biquasiprimitive actions on vertices, or are degenerate cycles
(see Table~\ref{tbl:basic} and the Framework discussion below). This new approach has been 
used before in other problems dealing with symmetries of graphs, and we believe that it will also bear fruit with regards to oriented graphs, and in particular, half-arc-transitive graphs of valency four.  
% In the mid-1960s three seminal research directions emerged in group theory related to 
% the symmetry of combinatorial or geometric structures. Jacques Tits \cite{Tits} in 1964 
% introduced BN-pairs. In the same year 
% Gert Sabidussi\cite{sab} introduced coset graphs as models of vertex-transitive graphs. 
% Several years later Donald Higman's `intersection matrices' paper \cite{dgh} established 
% a theory of orbital graphs to understand and analyse the structure of transitive permutation 
% groups. These important developments were inspired by quite different problems. Tits wished 
% to understand the structure of simple groups of Lie type geometrically  \cite{Tits2}; Sabidussi wished to 
% study graphs with various strong symmetry properties; while Higman was keen to understand 
% the structure of the new sporadic simple groups appearing in the 1960s, and indeed he 
% constructed a new simple group with Charles Sims by putting together certain actions of 
% a smaller sporadic simple group as orbitals of the new Higman--Sims group, see 
% \cite[Section 5.3]{dgh-obit}. The theme of this paper is Higman's orbital graphs.    

\subsection*{Oriented graphs.}
For a transitive permutation group $G$ on a set $X$, D. G. Higman~\cite{dgh} realised the importance of 
the $G$-action induced on ordered point-pairs, namely, for $g\in G$ and $x, y\in X$, 
$g:(x, y)\mapsto (x^g,y^g)$ where $x^g$ denotes the image of $x\in X$ under the action 
of $g$. Apart from the diagonal $\{(x,x)\mid x\in X\}$, each $G$-orbit $\Delta$ in this induced action corresponds a graph $\GD$ 
with vertex set $X$ admitting $G$ as a vertex-transitive and edge-transitive group of 
automorphisms. The edges of $\GD$ are the unordered pairs $\{x,y\}$ for which at least 
one of $(x,y)$ and $(y,x)$ lies in $\Delta$; such ordered pairs are called the \emph{arcs} of $\GD$. If both arcs $(x,y)$ and $(y,x)$ lie in 
$\Delta$ then $\Delta$ consists of all the arcs of $\GD$ and $G$ acts arc-transitively. 
All arc-transitive graphs arise in this way, and many classes of arc-transitive graphs have been studied intensively, such as distance 
transitive graphs \cite{PSY,vBon}, $s$-arc-transitive graphs \cite{montreal,seressBath}, 
and locally primitive and locally quasiprimitive graphs \cite{LPVZ,PPSS}. 
If $G$ is not arc-transitive then $\Delta$ consists of exactly one of $(x,y)$ and $(y,x)$ 
for each edge $\{x,y\}$ of $\GD$. Thus, by directing each edge $\{x,y\}$ from $x$ to $y$
if and only if  $(x,y)\in\Delta$, the orbital graph $\GD$ admits a $G$-invariant 
orientation of its edges and we say that the graph $\GD$ is \emph{$G$-oriented}.  
Notice that whenever $\GD$ is $G$-oriented, the graph $\GD$ has even \emph{valency}, 
say $m$, that is, each vertex $x$ lies in exactly $m$ edges, with $m/2$ of these edges 
$\{x,y\}$ directed from $x$ to $y$ and $m/2$ with the opposite orientation. All graphs 
that admit a $G$-invariant orientation, for some vertex- and edge- but not arc-transitive group $G$,
arise in this way. (The $G$-oriented graphs $\GD$ may also be regarded as directed graphs. However, we choose to view 
them as undirected graphs with an orientation induced by the group, both because this is 
the viewpoint taken in the literature mentioned below, and also since a given undirected graph may admit 
several interesting orientations corresponding to different groups, see for example, \cite{janc2}.) 
Since the group $G$ acts transitively on vertices, all connected components of $\GD$
are isomorphic, and the action induced by $G$ on each component is 
vertex-transitive, and edge-transitive, and preserves the $G$-invariant orientation. 
\emph{Thus we restrict attention to connected $G$-oriented graphs $\GD$.}
Let $\calF(m)$ denote the family of pairs $(\GD,G)$, where $\Delta$ is a $G$-orbit on ordered pairs for a transitive group $G$, and $\GD$ is connected and $G$-oriented
of valency $m$.
\medskip

Maru\v{s}i\v{c}~\cite[p.221]{Mar1} 
described progress up to 1998 on studying $\calF(4)$ as `thrilling'. One reason is undoubtedly a link 
to the study of maps on Riemann surfaces: research on this topic is led by Maru\v{s}i\v{c} and Nedela, 
and we sketch details of this link in Subsection~\ref{sec-medial}.    
Another reason is the fascinating internal structure discovered by Maru\v{s}i\v{c} connected 
with so-called \emph{alternating cycles}, which we discuss in Subsection~\ref{sec-prim}. 
They lead under certain conditions to a quotient of a pair $(\Gamma, G)\in\calF(4)$
which still lies in $\calF(4)$. However it is not clear how to describe the pairs
$(\Gamma, G)$ for which this procedure gives no `reduction'. 

We present here a new framework for studying the family $\calF(4)$
based on a theory of normal graph quotients, and the identification of a sub-family 
of \emph{basic} members of $\calF(4)$ with additional symmetry properties. 
This gives the potential of applying the theory of 
finite quasiprimitive permutation groups, and thereby exploiting the 
finite simple group classification to study these graphs.  
This has led us to new insights and new constructions of graphs in this family, 
to our knowledge not seen before in the literature. Our approach also provides a new 
framework for analysing known families of graphs in $\calF(4)$ (see below).  

\subsection*{Normal quotients.}
For a $G$-oriented graph $\Gamma$ with vertex set $X$, and a normal subgroup $N$ of $G$,
we define the \emph{$G$-normal quotient} $\Ga_N$ as follows: the 
vertex set is the set of $N$-orbits in $X$, and a pair $\{B,C\}$ 
of distinct $N$-orbits forms an edge of $\Ga_N$ if and only if  
there is at least one edge $\{x,y\}$ of $\Ga$ with $x\in B$ and 
$y\in C$. Note that $\Ga_N$ is defined as a graph 
with no specified orientation on its edges. 
Various degeneracies may occur when forming such quotients. For example,
$\Ga_N$ may consist of a single vertex if $N$ is vertex-transitive, and more generally, the valency of $\Ga_N$ may be a proper divisor of 
the valency  of $\Ga$. 
The graph 
$\Ga$ is called a \emph{$G$-normal $\ell$-multicover} of $\Ga_N$ if,
for each edge $\{B,C\}$ of $\Ga_N$, each vertex of $B$ is joined 
by an edge to exactly $\ell$ vertices of $C$ (or equivalently, 
each vertex of $C$ is joined by an edge to exactly $\ell$ vertices of $B$). In 
the case $\ell=1$ we usually say simply that $\Ga$ is a \emph{$G$-normal
cover} of $\Ga_N$. 
Our first result is an analogue for $G$-oriented graphs of the 
reduction theorem \cite[Theorem 1.1]{GP} of Gardiner and the 
fourth author for arc-transitive graphs of valency $4$. (A permutation group is \emph{semiregular}
if the only element fixing any point is the identity.)

\begin{theorem}\label{thm:nquot}
Let $(\Ga,G)\in\calF(4)$ with vertex set $X$, and let $N$ be a normal 
subgroup of $G$. Then $G$ induces a permutation group $\ov{G}$  
on the set of $N$-orbits in $X$, and either (i) $(\Ga_N,\ov{G})$ is 
also in $\calF(4)$, $\Ga$ is a $G$-normal cover of $\Ga_N$, $N$
is semiregular on vertices, and $\ov{G}=G/N$; or (ii) $(\Ga_N,\ov{G})$ is 
degenerate as in  one of the lines of 
Table~$\ref{tbl:nquot}$.
\end{theorem}

\begin{table}
\begin{center}
\begin{tabular}{cc|l}
\hline
&&\\
$\Ga_N$ & $\ov{G}$ & Comments \\ \hline
$K_1$               & $1$      & $N$ is vertex-transitive \\
$K_2$               & $Z_2$    & $\Ga$ is bipartite, $N$-orbits form the bipartition \\
$C_r$               & $D_{2r}$ & $r\geq3$, $\Ga$ is a $G$-normal $2$-multicover of $\Ga_N$\\
$C_r$ 		    & $Z_r$    & $r\geq3$, $\Ga$ is a $G$-normal $2$-multicover of $\Ga_N$ \\ \hline
\end{tabular}
\caption[DegenCases]{Degenerate cases for Theorem~\ref{thm:nquot}(ii)}
\label{tbl:nquot}
\end{center}
\end{table}

%$C_r^{\rightarrow}$

In Table~\ref{tbl:nquot}, the graphs $K_r, C_r$ denote a complete 
graph and a cycle on $r$ vertices, respectively. Moreover, the cycle $\Ga_N=C_r$ is $\ov G$-oriented if $\ov G = Z_r$,
and $\ov G$-arc-transitive if $\ov G = D_{2r}$; and we say that the cycle is oriented or unoriented, respectively.  
To set this theorem in a broader context we analyse the possibilities for $G$-normal 
quotients of $G$-oriented graphs of arbitrary valency in Proposition~\ref{prop:nquot} 
and derive Theorem~\ref{thm:nquot} from this result.  Theorem~\ref{thm:nquot}
shows that the family $\calF(4)$ is not closed under forming normal quotients as some
of the degenerate cases in Table~\ref{tbl:nquot} may arise. We call the pairs
$(\Ga_N,\ov{G})$ occurring in Table~\ref{tbl:nquot} \emph{degenerate pairs}. 
Clearly the first line with $\Ga_N=K_1$ occurs as a quotient for all $(\Ga,G)\in\calF(4)$
(take $N=G$), and the second line occurs whenever the graph $\Ga$ is bipartite (take $N$ the
index 2 subgroup of $G$ stabilising the two parts of the bipartition). 
The other two lines of Table~\ref{tbl:nquot} can also occur (see for instance, Example~\ref{ex:najat} 
for line 4, and see \cite{janc1,janc2} for more examples).

%%%%%%%%%%%

%

We call a pair $(\Ga,G)\in\calF(4)$  \emph{basic} if all of its $G$-normal quotients,
relative to nontrivial normal subgroups of $G$,
are degenerate pairs. Every $(\Ga,G)\in\calF(4)$
has at least one basic $G$-normal quotient $(\Ga_N,\ov{G})\in\calF(4)$ (Lemma~\ref{lem:basic1}), 
and our aim is to explore the possible kinds of basic pairs. It is helpful to 
subdivide them broadly as described in Table~\ref{tbl:basic}, since each type gives additional 
information about the group action. If the only normal quotients are as in line 1 of 
Table~\ref{tbl:nquot}, then all nontrivial normal subgroups of $G$ are transitive on the vertex set $X$; 
such groups $G$ are called \emph{quasiprimitive} (Table~\ref{tbl:basic}, line 1).
Similarly if  the only normal quotients are as in lines 1 or 2 of 
Table~\ref{tbl:nquot} and line 2 does occur, then 
every nontrivial normal subgroup of $G$  has at most two orbits in $X$, and 
at least one normal subgroup has two orbits; such a group $G$ is 
called \emph{biquasiprimitive} (Table~\ref{tbl:basic}, line 2).
For all other basic pairs there is at least one normal quotient as in line $3$ or $4$ of 
Table~\ref{tbl:nquot}, and the group $G$ has a dihedral or cyclic quotient, 
respectively (Table~\ref{tbl:basic}, line 3).

\begin{table}
\begin{center}
\begin{tabular}{lll}
\hline%\vspace{0.2cm}
Basic Type & Possible $\Ga_N$ for $1\ne N\vartriangleleft G$& Conditions on $G$-action \\ &&on vertices \\ \hline
Quasiprimitive & $K_1$ only & quasiprimitive \\
Biquasiprimitive & $K_1$ and $K_2$ only ($\Ga$ bipartite) & biquasiprimitive \\
Cycle & at least one $C_r$ ($r\geq3$) & at least one quotient action\\ && $D_{2r}$ or $Z_r$ \\ \hline
\end{tabular}
\caption[Basic]{Types of Basic pairs $(\Ga,G)$ in $\calF(4)$}
\label{tbl:basic}
\end{center}
\end{table}

\medskip\noindent
{\bf Framework for studying $\calF(4)$}.\quad  Theorem~\ref{thm:nquot} and 
the remarks above suggest a new framework for studying oriented graph-group 
pairs in $\calF(4)$, consisting of the following broad steps (see Table~\ref{tbl:basic}). 

\begin{enumerate}
\item Develop a theory to describe the quasiprimitive basic pairs in $\calF(4)$.
\item	Develop a theory to describe the biquasiprimitive basic pairs in $\calF(4)$.
\item	Develop a theory to describe the basic pairs in $\calF(4)$ of cycle type.
\item	Develop a theory to describe the $G$-normal covers $(X, G)\in\calF(4)$ of 
basic pairs of each of these three types. This theory should, for example, be 
powerful enough to describe, for a given basic pair $(Y, H)\in \calF(4)$, all 
pairs $(X, G) \in \calF(4)$ such that $X$ is a $G$-normal cover of $(Y, H)$. 
\item 	Apply this theory: for a given pair $(X, G) \in \calF(4)$, determine whether or not $(X, G)$ is basic. If it is non-basic, then find a basic $G$-normal quotient of it in $\calF(4)$ (or all of its basic $G$-normal quotients).
\end{enumerate}
 
Each of the steps in this framework requires delicate analysis. 
Substantial progress on completing the framework will provide a global 
structural view of the family $\calF(4)$. Regarding Step (4), there is a 
well-developed theory that will determine the normal covers $(X,G)$
of certain special types, for example, the normal quotients $Y=X_N$ modulo 
elementary abelian normal subgroups $N$ of $G$. The techniques used range from 
voltage assignments, representation theory, to studying universal covering groups,
see recent expositions of the general theory (with good discussions of the literature) in \cite{CM2,MMP},
and applications to $\calF(4)$ in \cite{CPS,PP}. 
A complete determination of all normal covers is probably not feasible.

In this paper we address Step (1) with short comments about Steps (2) and (3).
We make only a few brief comments on basic pairs of cycle type, 
as these pairs will be the theme of further work by the authors in \cite{janc1,janc2}. 
There are many infinite families of such graphs and we exhibit one such family in
Example~\ref{ex:najat}. 
A structure theorem is available to study quasiprimitive permutations groups
in \cite{qp} analogous to the O'Nan--Scott theorem for studying finite primitive 
permutation groups. Here we apply this theory to determine the possible types of
quasiprimitive groups $G$ that can arise for $(\Ga,G)\in\calF(4)$, that is,
$(\Ga,G)$ is a quasiprimitive basic pair (Step 1 of the framework). 
It would be interesting to study biquasiprimitive basic pairs in a simlar way.  
A group theoretic tool for this is available in \cite{biqp}, but is far 
less detailed than the quasiprimitive analogue and is more difficult to apply.

 \begin{problem}\label{prob:biqp}
Describe the biquasiprimitive basic graph-group pairs in $\calF(4)$.   
 \end{problem}

% Our objective in developing this framework for studying $\calF(4)$ has been that
% the basic graphs should admit a useful description, and enjoy additional properties 
% that facilitate their further study. This is certainly the case for the quasiprimitive basic pairs. 

\begin{theorem}\label{thm:basicqp}
Suppose that $(\Ga,G)\in\calF(4)$ is basic of quasiprimitive type.
Then $G$ has a unique minimal normal subgroup $N=T^k$ for some finite nonabelian simple group $T$, 
$k\leq2$, and one of the following holds.
 \begin{enumerate}
 \item[(a)] $k=1$, and there are many known examples; or
 \item[(b)] $k=2$, $G=N\cdot 2$, and $\Ga$ is a Cayley graph for $N$ as in Construction~$\ref{ex:tw}$; or
\item[(c)] $k=2$, $N$ is a not regular, and there are many known examples.
 \end{enumerate}
 \end{theorem}

\begin{remark}\label{rem:qp}{\rm
(i)  As commentary on this result we note that Li, Lu and 
Maru\v{s}i\v{c} \cite[Theorem 1.4]{LLM}
showed in 2004 that there are no vertex-primitive graph-group pairs in $\calF(4)$. 
This confirms the suitability of normal quotient reduction to quasiprimitive
actions as being the appropriate group theoretic reduction in the non-bipartite case. 

(ii) An infinite family of examples for Theorem~\ref{thm:basicqp}(a) with $T$
an alternating group  $\Alt(n)$ was given by Maru\v{s}i\v{c} in \cite{Mar}.
In these examples  $T$ is the full automorphism group and the vertex stabilisers  are elementary abelian of 
unboundedly large order as $n$ grows. In Construction~\ref{con:snbigstab} we give a 
similar construction using the groups $\Sym(n)$, which also have unbounded
vertex stabilisers. Our construction is slightly simpler than that in \cite{Mar} since we do not have the restriction that
 $\Sym(n)$ is the full automorphism group. 
In addition we provide two general constructions, namely Construction~\ref{ex:simpleCayley}
when $T$ is regular and Construction~\ref{con:simple} when $T$ 
is not regular. Both constructions rely on certain 2-generation 
properties of the simple group $T$, and a small concrete example 
is given of each construction. 

(iii) A general construction method is given in Construction~\ref{con:nonsimple}
for examples satisfying Theorem~\ref{thm:basicqp}(c), and a small concrete example is given.
This construction also depends on certain 2-generation properties of nonabelian simple groups.

(iv) Theorem~\ref{thm:basicqp} shows that exactly three of the eight 
types of quasiprimitive groups $G$ arise in basic pairs $(\Ga,G)\in \calF(4)$.
These types are sometimes called {\sc As, Tw, Pa} for cases (a)--(c) respectively. 
}
\end{remark}

% \begin{problem}\label{prob:qp}
% Decide if there are examples as in Theorem~$\ref{thm:basicqp}(c)$. 
% Determine if there are restrictions on the simple group $T$, or the integer $k$. 
% \end{problem}

% \begin{problem}\label{prob:qp2}
% Decide if there are examples in Theorem~$\ref{thm:basicqp}(c)$
% for every finite nonbelian simple group $T$. 
% Determine which examples from part (b) arise as the 
% $N$-normal quotient $(\Ga_M,T)$, and decide whether
% $k$ is bounded. 
% \end{problem}

\section{Brief comments on edge-transitive oriented graphs}\label{sec-lit} 

% in \texorpdfstring{$\calF(4)$}{OG(4)}

If the full automorphism group of a finite $G$-oriented vertex-transitive, edge-transitive graph $\GD$ preserves the orientation (and so does not act arc-transitive\-ly), then the 
graph $\GD$ is called \emph{$\frac{1}{2}$-transitive} or \emph{half-arc-transitive}. Also, a graph admitting $G$ as a vertex-transitive and edge-transitive, but not arc-transitive, group of automorphisms is sometimes referred to in the literature as \emph{$(G,\frac{1}{2})$-transitive}. For example, an $n$-cycle $C_n$, relative to the cyclic group $G=Z_n$, is $G$-oriented, vertex-transitive and edge-transitive, and hence is $(G,\frac{1}{2})$-transitive. However $C_n$ is not  
$\frac{1}{2}$-transitive since its full automorphism group $D_{2n}$ is 
arc-transitive. In 1966, Tutte~\cite{Tut1} asked whether any $\frac{1}{2}$-transitive 
graphs exist.  As any connected graph of valency $2$ is a cycle, it follows 
that all $\frac{1}{2}$-transitive graphs must necessarily have even valency at 
least $4$. In 1970 Bouwer~\cite[Proposition 2]{Bow} answered Tutte's question 
affirmatively by constructing, for each $k\geq2$, a $\frac{1}{2}$-transitive 
graph of valency $2k$ of order $6\cdot 9^{k-1}$ (that is to say, having 
$6\cdot 9^{k-1}$ vertices).

Thus, for each even integer $m\geq2$, $\calF(m)$ is the family of all pairs $(\Ga,G)$ 
such that $\Ga$ is a connected  $(G,\frac{1}{2})$-transitive graph of valency $m$. 
%That is to say, $G\leq\Aut(\Ga)$, $G$ is vertex-transitive, and $\Ga$ is a connected 
%orbital graph of valency $m$ which is $G$-oriented, see Subsection~\ref{sub-orbitalgraph} 
%for more details.  
Let $\calF:=\cup_m\calF(m)$. It is not difficult to 
construct pairs $(\Ga,G)$ in $\calF(m)$ for arbitrary $m$ using a lexicographic 
product construction. Moreover, Bower's work shows that, for each even $m$, $\calF(m)$ 
contains a  $\frac{1}{2}$-transitive pair $(\Ga,\Aut(\Ga))$. In the decades 
since Bower's work these graphs have been well-studied. The major pioneer 
in this work is Dragan Maru\v{s}i\v{c} and in 1998 Maru\v{s}i\v{c} published an  excellent 
survey~\cite{Mar1} of results and open problems up to that time. Summaries 
of more recent work are available in \cite{kmswx,MS}, and new advances appear regularly, 
such as the theory of alternets developed in \cite{hkd,Wil}.

As Maru\v{s}i\v{c} remarked in \cite{Mar1}, research on graph--group
pairs in $\calF$ took three main directions: (i) the search for, and study of,
pairs $(\Ga,G)$ in $\calF$ with $G$ primitive on vertices; (ii) classification 
of pairs $(\Ga,\Aut(\Ga))\in\calF$ of certain specified orders, for example, the 
order (number of vertices) being twice a prime, or a product of two primes, (or 
four times a prime~\cite{kmswx}),  etc; and (iii) an intensive study of the family 
$\calF(4)$. His paper \cite{Mar1} gives details of work in each of these three 
directions. We make a few brief comments here on the work of  
Maru\v{s}i\v{c} and others related to pairs $(\Ga,G)\in \calF(4)$.  In particular we discuss 
a different kind of quotienting operation for $\calF(4)$ (Section~\ref{sec-prim}), the way
such pairs arise from regular maps on surfaces (Section~\ref{sec-medial}), and
information on group structure making classification possible for orders up to 1,000  (Section~\ref{sec-stab}),

\subsection{The alternating cycles of \texorpdfstring{Maru\v{s}i\v{c}}{Marusic}}\label{sec-prim}

An alternating cycle in $\Ga$, where $(\Ga,G)\in\calF(4)$, is a cycle 
such that each pair of consecutive edges is oriented in opposite directions. 
Maru\v{s}i\v{c}~\cite[Proposition 2.4]{Mar2} showed that the alternating cycles all have the 
same even length and they partition the edge set of $\Ga$.
It is possible that there are only two alternating cycles and in this case 
Maru\v{s}i\v{c} proved that $\Ga$ belongs to an explicitly described family 
of circulant graphs \cite[Proposition 2.4]{Mar2}. If $\Ga$ contains more than 
two alternating cycles, then the non-empty intersections of the vertex sets of 
distinct alternating cycles have a fixed size,  called the \emph{attachment number} 
of $(\Ga, G)$, and form a system of blocks of imprimitivity for $G$. Moreover, 
the corresponding quotient graph is also a member of $\calF(4)$ relative to the 
induced $G$-action \cite[Theorem 1.1 and Theorem 3.6]{MP}. It is possible that 
the attachment number is $1$, and in this case the alternating cycles are 
said to be \emph{loosely attached}, the quotient is just the graph $\Ga$, 
and no reduction is achieved. However, for any positive integer $k$, there are
infinitely many examples of pairs $(\Ga,G)\in\calF(4)$ with attachment number $k$~\cite{MW}. 
Also if the attachment number is at least $3$, then a vertex stabiliser $G_x$ has size 
$2$ \cite[Lemma 3.5]{MP}. The attachment number is at most half the length 
of an alternating cycle. When this maximum is attained the cycles are said to be 
\emph{$G$-tightly attached}, and all pairs $(\Ga, G)$ in this 
case have been classified, in \cite[Theorem 3.4]{Mar2} and \cite[Lemma 
4.1 and Theorem 4.5]{MP} when the attachment number is odd and even, respectively. 
This classification has been `simplified and sharpened' by Wilson~\cite[Section 8]{Wil},
and all $\frac{1}{2}$-transitive examples identified by \v{S}parl~\cite{Spa}.

The possibility of understanding the internal structure of graph--group pairs in $\calF(m)$ for larger values of $m$, 
by identifying smaller quotients in the family is tantalising, and as far as we know has 
not been explored for general valencies $m$. Very recent work of Hujdurovi\'c~\cite{hkd}
studies vertex subsets of graphs in $\calF(m)$ called alternets, which are analogues for general $m$ of 
alternating cycles and were introduced in \cite[Section 4]{Wil}. They are equivalence classes of the so-called reachability relation introduced 
in \cite{cpw} for infinite arc-transitive digraphs. The paper \cite{hkd} focuses on oriented graphs with a small number of alternets.

%{\bf remove this para?}\quad
%The class of connected  edge-transitive $G$-oriented-digraphs of valency 2 
%(and their underlying valency 4 undirected graphs) is larger than the examples 
%that arise from medial graphs.  They have been studied extensively using 
%combinatorial and geometric methods, see for example \cite{AMN,Mar,MP,MS,S1,S2,S3,WF}, but their normal quotients have
%Snot been studied systematically. 

\subsection{Regular maps and their medial graphs}\label{sec-medial}

Cellular decompositions of surfaces are called \emph{maps}, and a common way to construct maps is 
by embedding a graph into a surface. For a map $M$ given by embedding a graph $\Sigma$ into some 
surface, the map group $\Aut M$ is the subgroup of automorphisms preserving the surface, and $M$ 
is called \emph{regular}
if $\Aut M$ acts transitively (and therefore regularly) on the set of arcs
of $\Sigma$. Regular maps are extensively studied in various branches of mathematics, including
combinatorics, Riemann surfaces and group theory \cite{Dyck, JS, Sz,Tut} going back to work of Tutte and Dyck,
with many recent papers inspired by the ground-breaking work of Jones and Singerman~\cite{JS}.

Some pairs $(\Ga,G)\in\calF(4)$ arise as medial graphs of regular maps,
and the work of Nedela and Marusic \cite{MN} 
established medial graphs/maps as a fundamental tool for studying regular maps on surfaces.
The medial
graph of a map $M$ (embedded into the same surface as $M$)  
is described as follows, see \cite[page 346]{MN} or \cite{Tut3}: subdivide each edge of 
$M$ with one new vertex. If $e$ and $f$ are two consecutive edges in a
boundary walk of a face $F$ of $M$, then join the two corresponding new vertices by a new edge
in $F$. Finally remove from the surface all the original vertices together with incident (old) edges. We
obtain a $4$-valent graph  ${\rm Med}(M)$ embedded into the surface, 
called the \emph{medial graph} of $M$. By its construction, it is embedded in the same 
surface as the original map $M$, and this embedded graph is called the medial
map of $M$. Maru\v{s}i\v{c} and Nedela \cite[Proposition 2.1]{MN} showed that  $({\rm Med}(M),G)
\in\calF(4)$, where $M$ is a regular map and $G$ is the map group,
and moreover vertex stabilisers $G_x$ have order $2$. They also proved a converse in 
\cite[Proposition 2.2]{MN}, namely that each $(\Ga,G)\in\calF(4)$ with vertex 
stabilisers $G_x$ of order $2$ arises as a medial map of some regular map.   $M$ with automorphism 
group $G$ such that ${\rm Med}(M)\cong\Ga$; and $M$ is determined up to duality 
and reflection. They also show  \cite[Theorem 4.1(4)]{MN} that for the special class of `negatively 
self-dual' regular maps $M$, the medial graph ${\rm Med}(M)$ admits a group $G$ which is twice 
as large as the map group of $M$, giving a second pair $({\rm Med}(M),G)
\in\calF(4)$, this time with vertex stabilisers $Z_2^2$. 

% %% Perhaps delete rest of paragraph? %%%
% For example, it follows from \cite{MX} that pairs $(\Ga,G)\in\calF(4)$ with $\Ga$ of girth $3$ are in 
% one-to-one correspondence with pairs $(\Sigma,G)$ where $\Sigma$ is a cubic graph admitting $G$ as a 
% subgroup of automorphisms acting regularly on arcs. In this correspondence  
% $\Sigma$ is the line graph of $\Ga$. However such a pair $(\Sigma, G)$ corresponds, by \cite[Theorem 1]{GNSS},
% to  a regular map $M$ and the original graph $\Ga$ is its medial graph ${\rm Med}(M)$. Nedela and Marusic 
% \cite[Section 3]{MN} generalised this interpretation, using maps and their medial graphs to give a correspondence between
% pairs $(\Sigma,G)$, with $\Sigma$ of valency $k$ and girth greater than $k$ and with $G$ regular on arcs, and pairs
% $(\Ga,G)\in\calF(4)$ with $\Ga$ of girth $k$. 
% %In a second paper~\cite{MN2} the graphs of girth 4 are characterised:
%which involves in particular, identifying from the list of tightly attached graphs in $\calF(4)$ those which contain a transitive $4$-cycle.  

\subsection{Vertex stabilisers}\label{sec-stab}
As mentioned above, pairs $(\Ga,G)\in\calF(4)$ with attachment number at least $3$ have 
vertex stabilisers of order $2$, and  the sizes of vertex stabilisers for graph--group pairs in $\calF(4)$ 
associated with medial graphs can be at most $4$, while the vertex stabilisers can be unboundedly 
large, even with $G$ a finite alternating group \cite[Theorem 1.1]{Mar}. 
In general the vertex stabilisers are $2$-groups of nilpotency class at most 2,
and they form the family of  \emph{concentric} $2$-groups, 
studied by Maru\v{s}i\v{c} and Nedela, see \cite[Sections 5--7]{MN2}, and first discovered by Glauberman \cite{Gla} by studying a rather different problem. (In fact, roughly speaking, Glauberman in \cite{Gla} investigates groups $G$ containing a finite $p$-subgroup $P$ and an element $g$ such that $G=\langle P,P^g\rangle$ with $|P:P\cap P^g|=p$. By taking $G$ to be a $\frac{1}{2}$-transitive subgroup of automorphisms of a $4$-valent graph, $P$ to be a vertex stabiliser $G_x$, and $g$ an element of $G$ mapping $x$ to one of its neighbours, we see that $G, P, g$ satisfy the hypotheses considered by Glauberman. Hence some of the results of Maru\v{s}i\v{c} and Nedela follow from the work of Glauberman.) 
The smallest nonabelian concentric $2$-group is the dihedral group $D_8$ of order $8$, and the first construction 
of a $\frac{1}{2}$-transitive graph with vertex stabilisers $D_8$ was given by Conder and  Maru\v{s}i\v{c} in \cite{CM}.
Their graph has $10,752$ vertices, and it was recently proved by Poto\v{c}nik and Po\v{z}ar \cite{PP} that there 
are exactly two $\frac{1}{2}$-transitive graphs of this order
and no such graphs with fewer vertices. Links between the stabiliser orders $|G_x|$ and the graph structure, have been studied in 
\cite{MN,PV,SV}. In particular, the information in \cite{SV} was sufficiently powerful to enable    
Poto\v{c}nik, Spiga and Verret~\cite{PSV} to classify all members of $\calF(4)$ 
with up to $1,000$ vertices.  It would be worth exploring realisations of concentric groups as vertex stabilisers in
basic pairs $(\Ga,G)\in\calF(4)$.
% One question raised by Nedela (2013, private communication) was whether the basic graphs of quasiprimitive
% or biquasiprimitive type might have vertex stabilisers of bounded order; but unfortunately
% the examples in Construction~\ref{con:snbigstab} show that this is not so.

\section{\texorpdfstring{$G$}{G}-oriented graphs and their normal quotients}

For fundamental graph theoretic concepts please refer to the book \cite{GR}. 
%We give here a detailed discussion of $G$-oriented edge-transitive graphs. 

\subsection{\texorpdfstring{$G$}{G}-oriented edge transitive graphs
% and \texorpdfstring{$G$}{G}-invariant edge orientations
}\label{sub-orbitalgraph} 

Suppose that $G\leq \Sym(X)$ is a transitive permutation group 
on $X$, and that $\Delta\subset X\times X$ is a non-diagonal $G$-orbit such that the associated graph 
$\GD$ is $G$-oriented and connected. As mentioned in the introduction $\GD$ admits $G$ as a 
vertex-transitive and edge-transitive group of automorphisms, and each edge $\{x,y\}$ is
oriented from $x$ to $y$ if and only if $(x,y)\in\Delta$.  Let 
$\calF$ denote the set of all such graph-group pairs $(\GD,G)$. 

The set $\Delta^*:=
\{(y,x) | (x,y)\in\Delta\}$ is also a $G$-orbit and is disjoint from $\Delta$ since $\GD$ 
is $G$-oriented. Moreover the graph $\mathcal{G}(\Delta^*)$ is $G$-oriented and has the same 
underlying undirected graph as $\GD$ with each edge oriented in the opposite direction. The 
two $G$-oriented graphs $\GD$ and $\mathcal{G}(\Delta^*)$ may, or may not, be isomorphic as oriented 
graphs: they are ismorphic  if and only if there is an automorphism $h$ of 
the underlying graph which maps some vertex-pair in $\Delta$ to a pair in $\Delta^*$.
Necessarily $h\not\in G$ since $G$ fixes $\Delta$ setwise, and such an automorphism exists 
if and only if $\GD$ is arc-transitive relative to a group larger than $G$. Note that it is not always 
possible to choose the automorphism $h$ to interchange the $G$-orbits $\Delta$ and $\Delta^*$, see 
\cite[Section 6]{CPS}, where it is claimed that the smallest graph $\Ga$ for which 
this is the case has $21$ vertices. We checked the possibilities and confirmed that 
there exists a unique $4$-valent, arc-transitive graph $\Ga$  with 21 vertices. It has automorphism 
group $\PGL(2,7)$ and dihedral vertex stabilisers of order $16$ (not of order 8 as claimed in \cite{CPS}).
There is a subgroup $G$ (a Frobenius group of order $42$) such that $(\Ga,G)\in\calF(4)$, and no automorphism of $\Ga$ 
interchanges the two $G$-orbits on arcs.

It follows from \cite{dgh,sab} that every graph or oriented graph admitting
$G$ as an edge-transitive and vertex-transitive group of automorphisms arises as
$\GD$ for some non-diagonal $G$-orbit $\Delta$ on ordered vertex-pairs. 
We note also that, by a result of Sims \cite[Proposition 3.1]{sims}, a $G$-oriented 
graph $\GD$ is strongly connected, in the sense that for every pair of distinct 
vertices there is an oriented path from the first to the 
second, if and only if the underlying undirected graph is connected (each pair 
of vertices is joined by a path in the graph with no restriction on the orientation of the edges 
in the path). 
  
For a positive integer $k$, let $\calAG(k)$ denote the set of graph--group pairs 
$(\Sigma,H)$ such that $\Sigma$ is a connected undirected graph of valency 
$k$, and $H\leq\Aut(\Sigma)$ acts arc-transitively.   Also let $\calAG:=\cup_{k\geq1}\calAG(k)$, and let $\calF:=\cup_{m \mathrm{even}}\calF(m)$,
where as in the introduction $\calF(m)$ is the set of pairs $(\Ga,G)$ such that $\Ga$ is a connected $G$-oriented graph of valency 
$m$, and $G\leq\Aut(\Ga)$ acts transitively on vertices and edges.

\subsection{Normal quotients of \texorpdfstring{$G$}{G}-oriented graphs}\label{sec-nquot}

Let $(\Ga,G)\in\calF$ with vertex set $X$ and of even valency $m\geq 4$, so $\Ga$ is connected and $\Ga=\GD$ for some non-diagonal $G$-orbit $\Delta$ in $X\times X$ such that $\Delta\ne\Delta^*$. Let $N$ be a normal subgroup of $G$, and recall the definition of the normal quotient graph $\Ga_N$ given 
in the introduction. Normal quotients were studied in 1989 by the fourth author 
\cite{P89} focussing on the special 
subfamily of $\calF$ consisting of pairs $(\Ga,G)$ such that 
$G$ is transitive on directed paths of length $2$ in $\Ga$
(see \cite[Section 3]{P89}, especially Lemma 3.2 and Theorem 3.3). Proposition~\ref{prop:nquot} is a generalisation of those results.

Since the $N$-orbits form a $G$-invariant 
partition of $X$ (see for example \cite[Lemma 10.1]{montreal}), and 
since $G$ is transitive on $X$ and on $\Delta$, it follows from the 
definitions of $\Ga=\GD$ and $\Ga_N$ %that $\Ga_N$ is connected and 
that $G$ induces a group $\ov{G}$ of automorphisms of $\Ga_N$ which 
is transitive on both vertices and edges.
One possibility is that $N$ has only one orbit on vertices
(for example, if $N=G$) and then $\Ga_N=K_1$ 
consists of a single vertex. In all other cases
$\Ga_N$ contains edges since $\Ga$ is connected, and it is
fairly easy to see that $\Ga_N$ must itself be connected.
It is possible that $\Ga_N$ inherits an
orientation from the $G$-orientation of $\Ga$, namely if, for some 
(and hence every) edge $\{B,C\}$ of $\Ga_N$, either 
all of the edges $\{x,y\}$ of $\Ga$ with $x\in B$ and $y\in C$ are such that $(x,y)\in\Delta$,
or all of these edges have $(x,y)\in\Delta^*$. In this case we prove in Proposition~\ref{prop:nquot} that $(\Ga_N,\ov{G})\in\calF$.
If $\Ga_N$ does not have this property, then  we show that every edge $\{B,C\}$ 
of $\Ga_N$ has $\Ga$-edges `in both directions' between vertices of $B$ and $C$, and $\ov{G}$ is arc-transitive
on  $\Ga_N$.

\begin{proposition}\label{prop:nquot}
Let $(\Ga,G)\in\calF(m)$ with vertex set $X$ and $m\geq4$, 
let $N$ be a normal subgroup of $G$. Then $\Ga_N$ is connected, $G$ induces a permutation group $\ov{G}$ on the set of $N$-orbits in $X$, 
and either $(\Ga_N,\ov{G})=(K_1,1)$,  
or one of the following holds.
\begin{enumerate}
\item[(a)] $(\Ga_N,\ov{G})\in\calF(k)$ and $\Ga$ is a $G$-normal $(m/k)$-multicover of $\Ga_N$, for some even divisor $k$ of $m$; or
\item[(b)] $(\Ga_N,\ov{G})\in\calAG(k)$ and $\Ga$ is a $G$-normal $(m/k)$-multicover of $\Ga_N$, 
for some divisor $k$ of $m$ with $m/k$ even (so $1\leq k \leq m/2$).
\end{enumerate}
\end{proposition}

\begin{remark}\label{rem:nquot}{\rm 
Consider Proposition~\ref{prop:nquot} case (b) with $k=1$. Here, 
for $x$ in an $N$-orbit $B$, say, all $m$ of the edges incident 
with $x$ have their other end point in a single second $N$-orbit, 
$C$ say, and by connectivity we find that there are just two $N$-orbits, 
so $\Ga$ is bipartite, the $N$-orbits form the bipartition, and $(\Ga_N, \ov{G})=(K_2, Z_2)$. 
}
\end{remark}

\begin{proof}
We use the notation and discussion given before the statement, so $\Ga=\GD$ 
for some $G$-orbit $\Delta\ne\Delta^*$ in $X\times X$, the quotient $\Ga_N$ is connected, and the induced group $\ov{G}$ is transitive on the vertices and edges of $\Ga_N$. 
Suppose that $\Ga_N\ne K_1$ so there are at least two $N$-orbits. Since 
$\Ga_N$ is connected, $\Ga_N$ has at least one edge $\{B,C\}$ and by 
definition there are vertices $x\in B, y\in C$, which form an edge of 
$\Ga$. Without loss of generality $(x,y)\in\Delta$. Let $\Delta(x)=\{z | 
(x,z)\in\Delta\}$, and let  $\ell_C(x)=|\Delta(x)\cap C|$.
Since $y\in \Delta(x)\cap C$, we have $\ell_C(x)>0$. Moreover, since 
$G$ is transitive on $\Delta$ it follows that $G_x$ is transitive on 
$\Delta(x)$,  and hence $\ell_{C'}(x)=\ell_C(x)=\ell$, say, for all 
$N$-orbits $C'$ that meet $\Delta(x)$ nontrivially. Thus $\ell$ divides 
$|\Delta(x)|$, and as the valency $m$ of $\Ga$ is equal to $|\Delta(x)|
+|\Delta^*(x)|=2|\Delta(x)|$, we have $2\ell \mid m$. Further, since $N$ 
is transitive on $B$ and fixes $C$ setwise, it follows that $\ell=\ell_C(x)$ 
is independent of the vertex $x$ of $B$ and hence the number of pairs 
from $B\times C$ lying in $\Delta$ is equal to $|B|\ell=|C|\ell$. Since 
also $N$ is transitive on $C$ and fixes $B$ setwise, we find that 
$\ell = |\Delta^*(z)\cap B|$ for each $z\in C$. 

Suppose first that there are no edges of $\Ga$ between $B$ 
and $C$ that are oriented `from $C$ to $B$', that is to say, 
$\Delta$ contains no pairs from $C\times B$. Since $\ov{G}$ is 
transitive on the edges of $\Ga_N$, this is true for all 
$\Ga_N$-edges, and hence $\Ga_N$ 
admits a $\ov{G}$-invariant orientation from the $G$-orientation of $\Ga$,
that is $(\Ga_N,\ov{G})\in\calF$. Moreover there are 
exactly $|B|\ell$ edges of $\Ga$ joining vertices in $B$ to 
vertices in $C$, and so $\Ga$ is a $G$-normal $\ell$-multicover of $\Ga_N$.      
Finally, since there are exactly $|B|m$ edges of $\Ga$ with one end in $B$,
it follows that the valency $k$ of $\Ga_N$ is $k=|B|m/(|B|\ell)=m/\ell$,
which is even since $2\ell \mid m$. Thus (a) holds.

Now suppose that there is also at least one $\Ga$-edge
from $C$ to $B$, that is, there is a pair $(c,b)\in\Delta$
with $c\in C, b\in B$. Since $G$ is transitive on $\Delta$, 
some element $g\in G$ maps $(x,y)$ to $(c,b)$, and hence
$g$ interchanges $B$ and $C$. This means that $\ov{G}$
acts arc-transitively on $\Ga_N$, so $(\Ga_N,\ov{G})\in\calAG$.
Further, $g$ maps the set $\Delta\cap (B\times C)$ to 
$\Delta\cap(C\times B)$, and it follows that each vertex $x$ of
$B$ is joined in $\Ga$ to $\ell$ vertices of $\Delta^*(x)\cap C$.
Thus there are $2\ell|B|$ edges of $\Ga$ between vertices of 
$B$ and $C$ (namely $\ell|B|$ in each orientation), so $\Ga$ is 
a $G$-normal $(2\ell)$-multicover of $\Ga_N$, and the valency $k$ of 
$\Ga_N$ is $k=m/2\ell\leq m/2$. So (b) holds.
\end{proof}

%

%Notice that the normal subgroup $N$ of $G$ acts trivially on the vertices of $\mathcal{G}(\Delta)_N$. We should more correctly refer to the orbital digraph 
%$\mathcal{G}(\Delta)_N$ being $G^{\mathcal{G}(\Delta)_N}$-edge transitive, etc, where $G^{\mathcal{G}(\Delta)_N}$ is the group of permutations of the $N$-orbits induced by $G$. For simplicity, we often write $G$-edge-transitive, etc, even when $G$ acts unfaithfully.

We derive Theorem~\ref{thm:nquot} from Proposition~\ref{prop:nquot}.

\medskip\noindent
\textbf{Proof of Theorem~\ref{thm:nquot}}\quad 
If $N$ is transitive on $X$ then $\Ga_N=K_1, \ov{G}=1$, and line 1 of Table~\ref{tbl:nquot} holds.
So suppose that $N$ is intransitive. Then $\Ga_N\ne K_1$ and so (a) or (b) of 
Proposition~\ref{prop:nquot} holds with $m=4$. In particular $\Ga_N$ is connected.
Suppose first that case (a) holds with $k=4$, so that $\Ga$ is a $G$-normal 
cover of $\Ga_N$. Since $N$
fixes each vertex of $\Ga_N$ setwise, $\ov{G}=G/K$ for some normal subgroup 
$K$ of $G$ containing $N$. Moreover the $K$-orbits are the same as the $N$-orbits. 
For each vertex $x$, the four vertices adjacent to $x$ in $\Gamma$ lie in four distinct
$N$-orbits, and hence the stabiliser $K_x$ fixes pointwise the vertices adjacent to $x$,
and by the connectivity of $\Ga$ it follows that $K_x=1$.  Thus $K$ is semiregular
on vertices,  since this holds for all vertices $x$. Also $|K|$ is
equal to the size of the $N$-orbit containing $x$, and this is a divisor of $|N|$. 
Hence $K=N$ and thus $\ov{G}=G/N$, and Theorem~\ref{thm:nquot}~(i) holds.

If case (a) of Proposition~\ref{prop:nquot} holds with $k=2$, then $\Ga$ is a $G$-normal $2$-multicover of $\Ga_N$,
and as $\Ga_N$ is an oriented connected $2$-valent graph it follows that $\Ga_N= \Cdr r$ 
and $\ov{G}=Z_r$, for some $r\geq3$, as in line 4 of Table~\ref{tbl:nquot}.

Suppose now that case (b) of Proposition~\ref{prop:nquot} holds. If $k=1$ then, as discussed
in Remark~\ref{rem:nquot}, $\Ga_N=K_2$, $\ov{G}=Z_2$, $\Ga$ is bipartite, 
and the $N$-orbits form the bipartition, as in line 2 of Table~\ref{tbl:nquot}.
If $k>1$ then the only possibility is $k=2$, and then $\Ga$ is a $G$-normal
$2$-multicover of $\Ga_N$, and $\Ga_N$ is a connected
$\ov{G}$-arc-transitive graph of valency $2$, so $\Ga_N=C_r$ and $\ov{G}=D_{2r}$,
for some $r\geq3$, as in line 3 of Table~\ref{tbl:nquot}.
\qed

\medskip

\subsection{Basic pairs in \texorpdfstring{$\calF(4)$}{OG(4)}}\label{sec-basic}

It is not difficult to see that every non-basic $(\Ga,G)\in\calF(4)$
has at least one basic normal quotient.

\begin{lemma}\label{lem:basic1}
For each non-basic $(\Ga,G)\in\calF(4)$, $\Ga$ is a $G$-normal cover of 
at least one quotient $\Ga_N$, with $(\Ga_N, G/N)\in\calF(4)$  basic.
\end{lemma}

\begin{proof}
By definition of a basic pair, since $(\Ga,G)\in\calF(4)$ is non-basic,
it follows from Theorem~\ref{thm:nquot} that there is a nontrivial 
normal subgroup $N$ of $G$ such that the quotient $(\Ga_N,\ov{G})
\in\calF(4)$ and $\Ga$ is a $G$-normal cover of $\Ga_N$, where $\ov{G}=G/N$. 

If $(\Ga_N,\ov{G})\in\calF(4)$ is basic there is nothing further to 
prove, so suppose that this is not the case. Then, as in the previous paragraph,
there is a nontrivial normal  subgroup $\ov{M}$ of $\ov{G}$ such that 
the quotient $((\Ga_N)_{\ov{M}},\ov{G}/\ov{M})
\in\calF(4)$ and $\Ga_N$ is a $\ov{G}$-normal cover of $(\Ga_N)_{\ov{M}}$.

Each such subgroup $\ov{M}$ is of the form $M/N$ for a (unique)
normal subgroup $M$ of $G$ properly containing $N$.
Moreover, it follows from the definitions that 
$(\Ga_N)_{\ov{M}}\cong\Ga_M$ and the group
$\ov{G}/\ov{M}=(G/N)/(M/N)\cong G/M$ so we have a larger normal subgroup
$M$ of $G$ and a smaller normal quotient $(\Ga_M,G/M)\in\calF(4)$,
and $\Ga$ is a $G$-normal quotient of $\Ga_M$. 
 
Since the group $G$ is finite this process can be applied a finite 
number of times, yielding a strictly increasing chain of normal 
subgroups of $G$, namely $1<N_1<N_2<\dots <N_s$, and a corresponding sequence of 
$G$-normal quotients $(\Ga_{N_1},G/N_1)$, $(\Ga_{N_2},G/N_2),\dots,
(\Ga_{N_s},G/N_s)\in\calF(4)$ such that
the final pair $(\Ga_{N_s},G/N_s)$ is basic, and $\Ga$ is a $G$-normal cover of $\Ga_{N_s}$. 
\end{proof}

As mentioned in Subsection~\ref{sec-stab}, for each $(\Ga,G)\in\calF(4)$, a vertex stabiliser 
 $G_x$ is a nontrivial $2$-group, and the 
possible orders $|G_x|$ for $(\Ga,G)\in\calF(4)$ are unbounded, even in the case of basic graphs. 
This was established in \cite{Mar} for the case where $G$ is an alternating group. 
We give here a simple example of basic cycle type, and we give further examples which are basic of quasiprimitive type in Construction~\ref{con:snbigstab}.

\begin{example}\label{ex:najat}
{\rm
Let $r\geq3$ and let $\Ga$ be the graph with vertex set $X=\mathbb{Z}_r\times \mathbb{Z}_2$ 
and edges $\{(i,j),(i+1,j')\}$ for all $i\in\mathbb{Z}_r, j,j'\in\mathbb{Z}_2$, 
that is, $\Ga= C_r[2.K_1]$, the lexicographic product of $C_r$ and $2.K_1$.
We orient each edge so that $(i,j) \rightarrow (i+1,j')$.
Let $G=Z_2\wr Z_r=\{(\sigma_1,\dots,\sigma_r)\tau^\ell |0\leq \ell<r\ \mbox{and each}\ \sigma_k\in \mathbb{Z}_2 \}$, 
where $(\sigma_1,\dots,\sigma_r):(i,j)\mapsto (i, j+\sigma_i)$, and $\tau:(i,j)\mapsto (i+1,j)$.
}
\end{example}

\begin{lemma}\label{lem:najat}
Let $r, \Ga, G, X$ and the edge orientation be as in Example~$\ref{ex:najat}$. Then $(\Ga, G)\in\calF(4)$, 
and for $x\in X$, $G_x\cong Z_2^{r-1}$. Moreover, each $(\Ga,G)$ is basic of cycle type.
\end{lemma}

\begin{proof}
By definition the graph $\Ga$ is a connected graph of valency 4.
It is straightforward to show that $G$ acts as a group of automorphisms of $\Ga$
which preserves the orientation, and is transitive on both vertices and edges, so
$(\Ga,G)\in\calF(4)$. As $G$ is vertex-transitive, without loss of generality we choose 
$x=(0,0)\in X$. The stabiliser is $G_x=\{(0,\sigma_1,\dots,\sigma_r)\mid 
\mbox{each}\ \sigma_k\in \mathbb{Z}_2\}\cong Z_2^{r-1}$. 

First we consider the quotient $\Ga_B$ modulo the `base group' 
\[
B=\{(\sigma_1,\dots,\sigma_r) |\ \mbox{each}\ 
\sigma_k\in \mathbb{Z}_2\}=Z_2^r. 
\]
This group is normal in $G$ and its orbits are
the sets $Y_i:=\{(i,0), (i,1)\}$, for $i\in\mathbb{Z}_r$. From the definition of
the edge orientation of $\Ga$, it follows that $\Ga_B$ is the cycle
$C_r$, and is $G/B$-oriented.

Next consider a minimal normal subgroup  $N$ of $G$. We claim that $N$ is contained
in $B$.  If this is not the case then $N\cap B=1$ and so $N\leq C_G(B)$.
However $C_G(B)=B$ and we have a contradiction. 
Thus $N\leq B$. Since all $B$-orbits have size 2, and since all $N$-orbits must have equal size
(because $N$ is normal in $G$), it follows that $N$ and $B$ have the same orbits, and hence
$\Ga_N=\Ga_B=C_r$. 

Now let $M$ be an arbitrary normal subgroup of $G$. Then $M$ contains a minimal normal
subgroup, and it follows from the previous paragraph that the $M$-orbits must be unions of
$B$-orbits. Thus the quotient $\Ga_M$ is isomorphic to a quotient of $\Ga_B$ 
(possibly equal to $\Ga_B$). Hence $\Ga_M$ is a cycle, or possibly $K_2$ or $K_1$, 
and so $(\Ga,G)$ is basic of cycle type.
\end{proof}

\section{Cayley graphs in \texorpdfstring{$\calF(4)$}{OG(4)}}\label{sec:cayley}

Many of the $\frac{1}{2}$-transitive graphs of valency four in the 
literature are Cayley graphs, and we need the following information about them for 
our analysis. 
\begin{enumerate}
 \item[(i)] For a group $N$ and inverse-closed subset $S$ of $N\setminus\{1_N\}$, 
the \emph{Cayley graph} ${\rm Cay}(N,S)$ has vertex set $N$, and edges $\{x,y\}$ such that $xy^{-1}\in S$.  
\item[(ii)] ${\rm Cay}(N,S)$ admits $N$ in its right multiplication action ($g:x\mapsto xg$ for $x,g\in N$)
as a subgroup of automorphisms that is \emph{regular} on vertices 
(that is, $N$ is 
transitive, and only the identity fixes a point).
\item[(iii)]  ${\rm Cay}(N,S)$ is connected
if and only if $S$ generates $N$.
\item[(iv)]  If $S=S_0\cup S_0^{-1}$ with $S_0\cap S_0^{-1}=\emptyset$, then for any pair 
$x, y$ of distinct vertices, at most one of $xy^{-1}$ and $yx^{-1}$ 
lies in $S_0$, and we may define an orientation on each  edge $\{x,y\}$ by 
$x\rightarrow y$ if and only if $yx^{-1}\in S_0$. 

\end{enumerate}

\begin{remark}\label{rem:rns}{\rm 
To explain how Cayley graphs arise naturally in our investigation,
we remark the following: if $G\leq\Aut(\Ga)$ for a 
graph $\Ga$ with vertex set $X$, and if $N$ 
is a normal subgroup of $G$ that is regular on $X$, then
we may identify $X$ with $N$ and 
there exists $S\subset N$ such that $\Ga = {\rm Cay}(N,S)$,
and $G$ is a semidirect product  $G=N\rtimes H$, where $N$ acts on $X=N$
by right multiplication, and $H\leq \Aut(N)$ acts naturally
as automorphisms fixing $S$ setwise. Moreover, $H$ 
is the stabiliser in $G$ of the vertex $1_N$, and the only element of $H$
fixing $S$ pointwise is the identity, since $S$ generates $N$.
See, for example, \cite[Section 1.7]{Ca}.
}
\end{remark}

\begin{lemma}\label{lem:cay}
Let $(\Ga,G)\in\calF(4)$, and let $N$ be 
a regular normal subgroup of $G$. Then $\Ga={\rm Cay}(N,S)$,
with $S=S_0\cup S_0^{-1}$ and $S_0=\{a,b\}$ such that $N=\langle S\rangle$ and  
$1_N\not\in\{a^2, b^2, ab\}$. Also $G=N\rtimes H$, where 
$H\leq \Aut(N)$, $H$ leaves $S_0$ invariant, and $H\cong Z_2$ interchanges 
$a$ and $b$.  
\end{lemma}

\begin{proof}
By Remark~\ref{rem:rns}, since $N$ is  
a regular normal subgroup of $G$, we have $\Ga={\rm Cay}(N,S)$,
and $G=N\rtimes H$, where $H$ is a subgroup 
of $\Aut(N)$ which leaves $S$ invariant, $H$ is the stabiliser of $1_N$,
and $H$ acts faithfully on $S$. Also, since $\Ga$ is connected, $N=\langle S\rangle$. 
Moreover, since $(\Ga,G)\in\calF(4)$, there are two distinct elements of 
$N$, say $a$ and $b$, such that $1_N\rightarrow a$ and $1_N\rightarrow b$.
Let $S_0=\{a,b\}$. Then by the definition of adjacency we have
$S_0\subset S$. The images of these oriented edges under the actions 
of $a^{-1}, b^{-1}\in N$ are  $a^{-1}\rightarrow 1_N$ and $b^{-1}\rightarrow 1_N$,
respectively. Thus also $S_0^{-1}=\{a^{-1},b^{-1}\}\subset S$. Since $1_N$ has exactly 
two `out-neighbours' and two `in-neighbours', it follows that 
$S=S_0\cup S_o^{-1}$ and $S_0\cap S_0^{-1}=\emptyset$. In particular
$a^{-1}, b^{-1}\not\in S_0$ and hence $1_N\not\in\{a^2, b^2, ab\}$. 
Since $G$ preserves the edge orientation and is edge-transitive,  
the stabiliser $H$ of $1_N$ must interchange $a$ and $b$, so $H\cong Z_2$.  
\end{proof}

In the case where $N$ is a regular minimal normal subgroup of $G$, 
the conditions of Lemma~\ref{lem:cay} are sufficient to exclude the possibility 
that $N$ is abelian (Lemma~\ref{lem:ha}), and to
allow us to construct examples with $N$ nonabelian 
(Constructions~\ref{ex:simpleCayley} and~\ref{ex:tw}).

\begin{lemma}\label{lem:ha}
If  $(\Ga,G)\in\calF(4)$ with  $\Ga={\rm Cay}(N,S)$ for 
a minimal normal subgroup  $N$ of $G$, then $N$ is not  abelian.
\end{lemma}
 
\begin{proof}
Suppose that $N$ is an abelian minimal normal subgroup of $G$. Then
$N=Z_p^d$ for some prime $p$ and integer $d\geq1$. By Lemma~\ref{lem:cay},
$S=S_0\cup S_0^{-1}$, $S_0=\{a,b\}$, and $G=N\rtimes H$ with $H=Z_2$ interchanging $a$ and $b$. 
This implies, since $N$ is abelian, that $H$ fixes $ab=ba$,
and by Lemma~\ref{lem:cay}, $ab\ne 1$. Since $N$ is abelian the subgroup
$\langle ab\rangle$ is therefore normal in $G$, and since $N$ is a minimal 
normal subgroup, this means that
$N= \langle ab\rangle$, so $d=1$. In this case, however, the only subgroup
$H$ of $\Aut(N)$ of order $2$ inverts $N$, and hence  
$b=a^{-1}$, which is a contradiction. 
\end{proof} 

\begin{construction}\label{ex:simpleCayley}{\rm
Let $N$ be a nonabelian simple group and $a\in N, \sigma\in\Aut(N)$,
such that $\sigma^2=1$, $S_0:=\{a,a^\sigma\}$ generates $N$, and let $S=S_0\cup S_0^{-1}$.
Then for $\Ga={\rm Cay}(N,S)$ and $G=N\rtimes \langle \sigma\rangle$,
we have $(\Ga,G)\in\calF(4)$.
}
\end{construction}

\begin{lemma}\label{lem:simpleCayley}
Each $(\Ga, G)$ in Construction~$\ref{ex:simpleCayley}$
lies in  $(\Ga,G)\in\calF(4)$ and is
basic of quasiprimitive type.
\end{lemma}

\begin{proof}
From our discussion above, $G\leq\Aut(\Ga)$, and $G$ is transitive on both the
vertex set $N$ and the edge set of $\Ga$. 
It remains to prove that $S_0\cap S_0^{-1}=\emptyset$, since from this it will
follow that $\Ga$ is $G$-oriented of valency $4$.
First we note that $a^{-1}\ne a$, since otherwise $S_0$ would consist of two involutions
and so cannot generate a nonabelian simple group. If $a^{-1}=a^\sigma$,
then $S_0$ generates a cyclic group. Thus $a^{-1}\not\in S_0$ and similarly 
$(a^\sigma)^{-1}\not\in S_0$. 
\end{proof}

For an example of a graph with the properties of 
Lemma~\ref{lem:simpleCayley}, take $N=\Alt(5)$, $a=(123)$, and
$\sigma\in\Aut(N)=\Sym(5)$ the inner 
automorphism induced by conjugation by $(14)(25)$. 
It would be interesting to have a good understanding of
the family of Cayley graphs arising in Lemma~\ref{lem:simpleCayley}.
In particular generation results such as those in \cite{LiSh} 
should give some insights.

\begin{problem}\label{prob:simplecayley}
Determine all the Cayley graphs $\Ga$ for nonabelian simple
groups $T$ such that $(\Ga,G)\in\calF(4)$ for some $G\leq\Aut(T)$.
\end{problem}

\begin{construction}\label{ex:tw}
{\rm
Let $T$ be a nonabelian simple group, and let
$\{a,b\}$ be a generating set for $T$ such that
no automorphism of $T$ interchanges $a$ and $b$.
Let $N=T\times T$, $S_0=\{(a,b), (b,a)\}$, $S=S_0\cup S_0^{-1}$, and 
$\Ga={\rm Cay}(N,S)$. Let $\tau\in\Aut(N)$
be the map $\tau:(x,y)\mapsto(y,x)$, and let
$G=N\rtimes \langle\tau\rangle$.
}
\end{construction}

All nonabelian simple groups
have many generating pairs satisfying
the requried conditions, (see \cite{LiSh}).
For example if $T=\Alt(5)$, we could take $a=(123)$
and $b=(12345)$.

\begin{lemma}\label{lem:tw}
Each $(\Ga, G)$ in Construction~$\ref{ex:tw}$
lies in  $\calF(4)$ and is
basic of quasi\-primitive type.
\end{lemma}

\begin{proof}
First we note that $\tau$ interchanges the two
elements of $S_0$. Consider the subgroup 
$N_0=\langle S\rangle$. The projections of $N_0$
to the first and second direct factors of
$N$ are both equal to the group $\langle a,b\rangle = T$,
since $\{a,b\}$ is a generating set for $T$.
Hence either $N_0=N$ or $N_0$ is a diagonal subgroup of $N$.
In the latter case there exists $\sigma\in\Aut(T)$ such that
$N_0=\{ (x,x^\sigma) | x\in T\}$. In particular $(a,b)=(a,a^\sigma)$
so $b=a^\sigma$, and also $(b,a)=(b,b^\sigma)$ so $a=b^\sigma$. 
However by assumption, no such automorphism exists.
Hence $N_0=N$, and this implies that $\Ga$ is connected.

We claim that $S_0\cap S_0^{-1}=\emptyset$.
It is sufficient to prove that $x:=(a,b)^{-1}\not\in S_0$. 
If $x=(a,b)$ then $a^2=b^2=1$ and $\langle a,b\rangle\ne T$. 
If  $x=(b,a)$ then $b=a^{-1}$ and again $\langle a,b\rangle\ne T$. 
So the claim is proved and hence $\Ga$ 
has valency $4$, and admits an edge orientation as described above. 
Then, since $N$ acts regularly on vertices by 
right multiplication, and since $\tau$ interchanges
the two elements of $S_0$, it follows that $G$ is vertex-transitive
and edge-transitive, and preserves the edge orientation. Hence $(\Ga,G)\in\calF(4)$. 

Finally, since $\tau$ acting by conjugation, 
interchanges the two simple direct factors of $N$,
the group $N$ is a minimal normal subgroup of $G$,
and is the unique minimal normal subgroup. Then since
$N$ is regular on vertices, we conclude that $G$ is 
quasiprimitive on the vertex set of $\Ga$. Hence 
$(\Ga,G)$ is basic of quasiprimitive type.
\end{proof}

We remark that the quasiprimitive group $G$ in Construction~\ref{ex:tw}
is of twisted wreath type ({\sc Tw}), as defined in \cite{qp}. 

\section{Coset actions and coset graphs in \texorpdfstring{$\calF(4)$}{OG(4)}}\label{sec:coset}

We obtain further examples of graphs in $\calF(4)$ using the coset graph construction.
For a group $G$, proper subgroup $H$, and element $s\in G$, the \emph{coset graph}
$\Ga={\rm Cos}(G,H,s)$ is the (undirected) graph with vertex set 
$V=\{Hx\ \mid\, x\in G\}$, and edges $\{Hx,Hy\}$ if and only if $xy^{-1}$ or $yx^{-1}\in HsH$.
A good account of this construction is given in \cite[Section 2]{LLM}. 
The group $G$, acting by right multiplication on $V$, induces a vertex-transitive 
and edge-transitive group of automorphisms of $\Ga$. This $G$-action is faithful, so that
$G\leq\Aut(\Ga)$, if and only if $H$ is \emph{core-free}, that is, $\cap_{g\in G}
H^g=1$. The graph $\Ga$ is $G$-oriented if and only if $s^{-1}\not\in HsH$ (and 
otherwise $\Ga$ is $G$-arc-transitive). In the $G$-oriented case the valency is
$2|H:H\cap H^s|$. Also $\Ga$ is connected if and only if 
$\langle H, s\rangle = G$. \emph{In summary, $(\Ga, G)\in\calF(4)$ if and only if}

\begin{equation}\label{eq:coset} 
\mbox{\emph{$H$ is core-free in $G$, $s^{-1}\not\in HsH$, $|H:H\cap H^s|=2$, and $\langle H, s\rangle = G$.}}
\end{equation}

Moreover, for each $(\Ga,G)\in\calF(4)$, $\Ga = {\rm Cos}(G,H,s)$ for some $H, s$
satisfying \eqref{eq:coset}.
First we give a family of examples for nonabelian simple groups.

\begin{construction}\label{con:simple}
{\rm
Let $G$ be a nonabelian simple group,  and suppose that $G$ contains
an element $h$ of order $2$, and an element $g$ such that $\{g,g^h\}$
generates $G$. Let $\Ga={\rm Cos}(G,\langle h\rangle,g)$. 
}
\end{construction}

\begin{lemma}\label{lem:simple}
For $\Ga, G$ as in Construction~$\ref{con:simple}$,
$(\Ga,G)\in\calF(4)$ and is basic of quasiprimitive type.
\end{lemma}

\begin{proof}
Each coset graph for $G$ is $G$-vertex-transitive and $G$-edge-transitive.
Let $H:=\langle h\rangle$. 
To see that $\Ga$ is $G$-oriented observe that
$g^{-1}\not\in HgH=\{g,hg,gh,hgh\}$: for if $g^{-1}=hg$ or $gh$
then $h=g^{-2}$ and so $\langle g,g^h\rangle$ is cyclic, and we
have the same conclusion if $g^{-1}=hgh$; and if $g^{-1}=g$ then
$\langle g,g^h\rangle$ is generated by two involutions and 
so is abelian or dihedral.   

The fact that $\langle g,g^h\rangle=G$ also implies that $g^h\ne g$,
and hence $h^{g}\ne h$, so the stabiliser $H^g$ of the vertex $Hg$ adjacent to 
$H$ is not equal to $H$. Hence $H^g\cap H=1$ and it follows that
from \eqref{eq:coset} that $(\Ga,G)\in\calF(4)$. Since $G$ is a nonabelian simple group,
the pair is basic of quasiprimitive type.
\end{proof}

Now we construct a family of coset graphs for the 
symmetric groups for which the stabilisers are unbounded.
Our construction is similar in spirit to constructions of 
arc-transitive coset graphs for $\Sym(n)$ in \cite{CW,PSV11}, and 
supplements the const
ruction in \cite{Mar} for the groups $\Alt(n)$
(see Remark~\ref{rem:qp}(ii)).

\begin{construction}\label{con:snbigstab}
{\rm
Let $G=\Sym(n)$ with $n$ odd and $n\geq5$, let $m:=(n-1)/2$, let 
$H=\langle (i,i+m) \mid i=1,\dots,m\rangle$, and $g=(1,2,3\dots,n)$.
Let $\Ga={\rm Cos}(G,H,g)$. 
}
\end{construction}

\begin{lemma}\label{lem:snbigstab}
Let $\Ga, G$ be as in Construction~$\ref{con:snbigstab}$.
Then $(\Ga,G)\in\calF(4)$ is basic of quasiprimitive type
and a vertex stabiliser $G_x$ has order $2^{(n-1)/2}$.
\end{lemma}

\begin{proof}
The stabiliser of the vertex $x=H$ of $\Ga$ is $G_x=H$ of order
$2^{(n-1)/2}$. To prove that $(\Ga, G)\in\calF(4)$, we check 
that the conditions \eqref{eq:coset} hold. The only nontrivial normal
subgroups of $G$ are $\Alt(n)$ and $G$ itself, and neither of these 
is contained in $H$, so $H$ is core-free in $G$.
Also $H^g =  \langle (i,i+m) \mid i=2,\dots,m+1\rangle$,
and hence $H\cap H^g = \langle (i,i+m) \mid i=2,\dots,m\rangle$ has
index $2$ in $H$.

We claim that $g^{-1}\not\in HgH$: for if $g^{-1}\in HgH$ then we would have $ghg\in H$ for some $h\in H$.
Since each element of $H$ fixes the point $n$ (in its 
natural action on $\{1,2,\dots,n\}$), $n=n^{ghg}=1^{hg}=(1^h)^g$ so $1^h=n-1$.
However the $H$-orbit containing $1$ is $\{1,m+1\}$, so this is a contradiction since $n\geq5$.

Finally the group  $\langle H,g\rangle$ 
is transitive on $\{1,\dots,n\}$, since it contains $g$, and we show that this action is primitive. 
Now  $\langle H,g\rangle$ 
contains $(i,i+m)$, for each $i$, and in particular it contains $(1,\frac{n+1}{2}),
(\frac{n+1}{2},n)$, $(n,\frac{n-1}{2}), (\frac{n-1}{2},n-1), \dots$. Thus any block of imprimitivity $B$ of 
size at least $2$, and such that $1\in B$, must also contain 
$\frac{n+1}{2}, n, \frac{n-1}{2}, n-1$, and so on, it follows that $B=\{1,\dots,n\}$.
Hence $\langle H,g\rangle$ is primitive on $\{1,\dots,n\}$, and since 
$\langle H,g\rangle$ contains a transposition, it must be equal to $\Sym(n)$. Thus 
 the conditions \eqref{eq:coset} hold, and so $(\Ga,G)\in\calF(4)$.

The only proper nontrivial normal subgroup of $G$ is $\Alt(n)$, and since
$H$ contains odd permutations, it follows that $H\Alt(n)=\Sym(n)$, so $\Alt(n)$ 
is vertex-transitive. Thus $G$ is quasiprimitive on vertices, and so 
$(\Ga,G)$ is basic of quasiprimitive type.
\end{proof}

Finally we give a construction for quasiprimitive groups with 
non-simple, non-abelian minimal normal subgroups as in Theorem~\ref{thm:basicqp}(c).

\begin{construction}\label{con:nonsimple}
{\rm
Let $T$ be a non-abelian 
simple group such that $T=\langle a, b \rangle$, where $a$ is an involution and 
$b^g\ne ba$ for any $g\in \cent {\Aut(T)} a$.
Consider $G:=N\rtimes \langle\iota\rangle$, where $N=T\times T$ and $\iota:N\to N$ is the automorphism 
defined by $(x,y)^\iota=(y,x)$ for every $(x,y)\in N$. Let $H=\langle (a,a),\iota\rangle$, 
$g=(b,ba)$, and $\Ga={\rm Cos}(G,H,g)$. 
}
\end{construction}

Explicit examples of $T$, $a$ and $b$ are easy to obtain: 
for example $T=\mathrm{Alt}(5)$, $a=(1,2)(3,4)$ and $b=(1,5,4,3,2)$.

\begin{lemma}\label{lem:nonsimple}
Let $\Ga, G, N, H$ be as in Construction~$\ref{con:nonsimple}$.
Then $(\Ga,G)\in\calF(4)$ is basic of quasiprimitive type,
$N$ is the unique minimal normal subgroup of $G$, and
$H=\mathbb{Z}_2\times\mathbb{Z}_2$.
\end{lemma}

\begin{proof}
We verify that the conditions in \eqref{eq:coset} hold. First it is straightforward to
check that $H=\mathbb{Z}_2\times\mathbb{Z}_2$ and $H$ is core-free in $G$. Next an easy computation shows that
$$
H^{g}\cap H=\langle (a,a)\iota\rangle
$$
and hence $H^{g}\cap H$ has index $2$ in $H$.
To prove that $G=\langle H,g\rangle$, it is sufficient to prove 
that $N=\langle (a,a),g\rangle$, and the latter follows easily from the facts that 
$T=\langle a,b\rangle=\langle a,ba\rangle$ and that $b$ and $ba$ are 
not conjugate via an element of $\cent {\Aut(T)} a$. Thus $G=\langle H,g\rangle$. 
Finally a straightforward check shows that, for none of the 16 pairs $(h,h')$ in $H\times H$,
do we have $g^{-1}=hgh'$. Hence $g^{-1}\not\in HgH$. Thus all  the conditions in \eqref{eq:coset} hold,
and so $(\Ga,G)\in\calF(4)$.

Since $\iota$ interchanges the two simple direct factors of $N$, it follows that $N$ is a minimal normal subgroup of $G$, 
and as $N$ has trivial centraliser in $G$, it must be the unique minimal normal subgroup. Since $G=NH$, the group $N$ is 
vertex-transitive, and so $G$ is quasiprimitive on vertices. Thus the only $G$-normal quotient of $(\Ga,G)$ is 
the one-vertex graph $K_1$, so $(\Ga,G)$ is basic of quasiprimitive type.  
\end{proof}

\section{Proof of Theorem~\ref{thm:basicqp}}

Suppose that $(\Ga,G)\in\calF(4)$ is basic of quasiprimitive type,
with vertex set $X$. Let $N$ be a minimal normal subgroup of $G$.
Then $N$ is transitive on $X$ since $G$ is quasiprimitive. 
If $N$ is abelian then $N$ is regular (see, \cite[Theorem 1]{qp}) and so
by Remark~\ref{rem:rns},
$\Ga={\rm Cay}(N,S)$ for some $S$. This is not possible by
Lemma~\ref{lem:ha}. Thus $N$ is nonabelian, and so
$N\cong T^k$ for some nonabelian simple group $T$, and 
integer $k\geq1$. Moreover, since a stabiliser $G_x$ is a nontrivial $2$-group,
it follows from \cite{qp} that $N$ is the
unique minimal normal subgroup of $G$, and hence $G$ is isomorphic 
to a subgroup of $\Aut(N)=\Aut(T)\wr \Sym(k)$, where $N$ is identified
with its group of inner automorphisms. If $k=1$ then part (a) holds.
So we may assume that $k\geq2$. By \cite[Theorem 1]{qp}, and using the fact that
$G_x$ is a nontrivial $2$-group, there are the following two  possibilities. 
(A \emph{subdirect subgroup} of $R^k$, where $k\geq2$, is a subgroup
$H$ for which each of the $k$ projections of $H$ onto the $k$ direct factors $R$ 
has image equal to $R$.)

\begin{enumerate}
\item[(i)] $N$ is regular on $X$ (the twisted wreath case of `type III(c)');
\item[(ii)] $N_x$ is a subdirect subgroup of $R^k$ for some nontrivial $2$-subgroup $R$ of $T$ 
(the product case of `type III(b)(i)'). 
\end{enumerate}
We treat these two cases (i) and (ii) in Lemmas~\ref{lem:twrproof}
and~\ref{lem:nonsimple}, respectively. The proof of Theorem~\ref{thm:basicqp} follows from these two lemmas.

Since $N$ is a minimal normal subgroup, $G$ permutes the $k$ simple 
direct factors of $N$ transitively by conjugation. Moreover, since $N$ is vertex-transitive we have
$G=NG_x$, and hence the $2$-group $G_x$ also permutes the $k$ simple 
direct factors transitively. In particular $k$ divides $|G_x|$. 

\begin{lemma}\label{lem:twrproof}
If $N=T^k$ is regular with $k\geq2$,
then $k=2$, the pair $(\Ga,G)$ is as in Construction~$\ref{ex:tw}$,
and Theorem~$\ref{thm:basicqp}\,(b)$ holds.
\end{lemma}

\begin{proof}
Since $N$ is regular, $N_x=1$. Then, again by
Remark~\ref{rem:rns}, $\Ga={\rm Cay}(N,S)$ with $S=S_0\cup S_0^{-1}$, and $G=N\rtimes H$, 
where $H$ is a subgroup of $\Aut(N)$ which leaves $S_0$ invariant.
By Lemma~\ref{lem:cay},
$S_0=\{a,b\}$ such that $N=\langle S\rangle$ and none of 
$a^2, b^2, ab$ is $1_N$. Also $H\cong Z_2$ and interchanges $a$ and $b$.
Since $H$ is the stabiliser of $x=1_N$, we saw above that $k$ divides $|H|$, and hence 
$k=2$.  Then $a=(a_1,a_2), b=(b_1,b_2)\in N= T\times T$, and $H=\langle h\rangle$,
where $h=(\sigma, \sigma')\tau\in\Aut(T)\wr \Sym(2)$, with
$\sigma,\sigma'\in\Aut(T)$ and $\tau=(1,2)\in\Sym(2)$. We require
$h^2=1$ and as $h^2=((\sigma,\sigma')\tau)^2 = (\sigma\sigma',\sigma'\sigma)$,
it follows that $\sigma'=\sigma^{-1}$. Next we observe that 
the element  $(\sigma,1)\in\Aut(N)\leq\Sym(N)$ induces a graph isomorphism
from $\Ga$ to the graph ${\rm Cay}(N,S^{(\sigma,1)})$ admitting the edge-transitive group
$G^{(\sigma,1)}=N\rtimes\langle h^{(\sigma,1)}\rangle$. Now 
\[
h^{(\sigma,1)}=(\sigma^{-1},1)(\sigma,\sigma^{-1})\tau (\sigma,1) =
(1,\sigma^{-1})\tau (\sigma,1) = \tau
\]
so, replacing $(\Ga,G)$ by their images under $(\sigma,1)$, if necessary,
we may assume that $\sigma=1$, that is, $h=\tau$. Thus $G$ is as in
Construction~\ref{ex:tw}. Now $b=a^h = (a_1,a_2)^\tau = (a_2,a_1)$. 
Since $\Ga$ is connected $N=\langle a,b\rangle$. In particular 
this implies that $\{a_1,a_2\}$ generates $T$. If there were an 
automorphism $\nu\in\Aut(T)$ that interchanged $a_1$ and $a_2$, then
we would find that $\langle a,b\rangle = \{(t,t^\nu)\mid t\in T\}$,
which is a contradiction. Hence no such automorphism exists,
and $(\Ga,G)$ is as in Construction~\ref{ex:tw}. Thus part (b)
of Theorem~\ref{thm:basicqp} holds.
\end{proof}

Now we consider the second case above. The analysis follows closely 
some ideas developed in~\cite{sims,Sims1,Tutte,Tutte2}. Note that examples for this case were given in Construction~\ref{con:nonsimple}.

\begin{lemma}\label{lem:pa}
Suppose that $N=T^k$ with $k\geq2$, and $N_x$ is a subdirect subgroup 
of $R^k$ for some nontrivial $2$-subgroup $R$ of $T$.  Then $k=2$,
and Theorem~$\ref{thm:basicqp}\,(c)$ holds. 
\end{lemma}

\begin{proof}
As discussed in Subsection~\ref{sub-orbitalgraph}, $\Ga=\GD$ for some $G$-orbital $\Delta$.
Let $s$ be a positive integer. By an $s$-arc of $(\Ga,G)$ we mean   
a vertex-sequence $(x_0,\ldots,x_s)$ such that $(x_i,x_{i+1})\in\Delta$ 
for each $i\in \{0,\ldots,s-1\}$. Now, let $s$ be maximal such 
that $G$ is transitive on the set of $s$-arcs. Note that $s\geq 1$ since $(\Ga,G)\in\calF(4)$.
We claim that $G$ acts regularly on the set of 
$s$-arcs of $\Gamma$. Let $(x_0,\ldots,x_s)$ be an $s$-arc and consider  
the point-wise stabilizer $H$ in $G$ of $x_0,\ldots,x_s$. 
To prove the claim it suffices to prove that $H=1$. 
Since $G$ is not transitive on $(s+1)$-arcs, the group $H$ fixes each of the two vertices $y, y'$ such that $(x_s, y), (x_s,y')\in\Delta$. 
(Let us call $y, y'$ the \emph{out-neighbours} of $x_s$.)
Now $(x_1,\ldots,x_{s},y)$ is an $s$-arc and so, by transitivity, there exists $g\in G$ 
such that $(x_0,\ldots,x_{s-1},x_s)^g=(x_1,\ldots,x_{s},y)$. 
Therefore $H^g=G_{x_0,\ldots,x_s}^g=G_{x_1,\ldots,x_{s},y}$. However, since $H$ fixes $y$ it follows that  
$G_{x_1,\ldots,x_{s},y}\leq H$, and hence $H^g=H$. The fact that $H$ fixes the two out-neighbours 
of $x_s$ implies that $H^g$ fixes the two out-neighbours of $x_s^g=y$, that is to say, $H=H^g$ fixes also the 
two out-neighbours of $y$. At this point 
it is clear that a quick inductive argument (using the connectedness of $\Gamma$) yields that $H=1$.
Thus the claim is proved.

Since $G$ is regular on $s$-arcs, its subgroup $G_{x_0}$ is regular on the set of $s$-arcs starting 
at $x_0$. From this it quickly follows that, for each $i\in \{0,\ldots,s\}$, the subgroup $G_{x_0,\ldots,x_{s-i}}$ has order 
$2^i$. In particular, $|G_{x_0,\ldots,G_{x_{s-1}}}|=2$, say $G_{x_0,\ldots,G_{x_{s-1}}}=\langle h_1\rangle$. Let $g\in G$
be as in the previous paragraph and set $x_{s+1}=y$.  
For $i\in \{2,\ldots,s\}$, define $h_i=h_{i-1}^{g^{-1}}$. It is clear that, for each $i$, we have 
$$
G_{x_0,\ldots,x_{s-i}}=\langle h_1,\ldots,h_i\rangle.
$$
As elements of  $\Aut(N)=\Aut(T)\wr \Sym(k)$, $h_1$ and $g$ may be written as 
$h_1=f\sigma$ and $g=f'\tau$ with $f,f'\in \Aut(T)^k$ and $\sigma,\tau\in \Sym(k)$. 
Note that, since $h_1^2=1$, we also have $\sigma^2=1$.
Let $\pi$ denote the projection map $\Aut(N)\rightarrow \Sym(k)$, so that $(h_1)\pi=\sigma$, and $(g)\pi=\tau$. 
Now $K:=(G)\pi=(NG_{x_0})\pi=(G_{x_0})\pi$ is a $2$-group since $G_{x_0}$ is a $2$-group. Moreover,
\begin{equation}\label{dagger}
K=(G_{x_0})\pi=\langle h_1,h_2,\ldots,h_s\rangle\pi=\langle \sigma,\sigma^{\tau^{-1}},\sigma^{\tau^{-2}},\ldots,\sigma^{\tau^{-(s-1)}}\rangle. 
\end{equation}
We claim that $K=\langle \sigma\rangle$, from which it follows that $k=2$ since $\sigma^2=1$. 
Suppose to the contrary that $\langle\sigma\rangle$ is a proper subgroup of $K$, and let  
$M$ be a maximal subgroup of $K$ containing $\sigma$. As $K$ is a $2$-group, $M$ is normal 
in $K$ and hence $\sigma^{\tau^{-i}}\in M$,  for all $i\in \mathbb{Z}$. It then follows from \eqref{dagger} that 
$K\leq M$, whence $K=M$, which is a contradiction. Thus $K=\langle \sigma\rangle\cong \mathbb{Z}_2$ and $k=2$.
\end{proof}

\section*{Acknowledgements}

The authors are grateful to Gabriel Verret for reading the draft manuscript and for
his advice on exposition and related work. We thank Roman Nedela for his interest in 
our research and advice on references. We thank an anonymous referee for helpful feedback. 
This project was funded by the Deanship of Scientific 
Research (DSR), King Abdulaziz University, Jeddah, under grant no. HiCi/H1433/363-1. The 
authors, therefore, acknowledge with thanks DSR technical and financial support.

\end{document}